\documentclass[11pt]{amsart}
\usepackage{graphicx, times}
\usepackage{amsmath,amsthm,amsfonts,amssymb,amscd,amsbsy}
\usepackage{caption,enumerate}
\usepackage{dsfont,lscape}
\usepackage{hyperref}
\usepackage{cleveref}

\usepackage{tikz}
\tikzset{>=latex} % for LaTeX arrow head
\usepackage{xcolor}
\colorlet{myblue}{black!20!blue}
\colorlet{myred}{black!20!red}

\hypersetup{
    pdftoolbar=true,
    pdfmenubar=true,
    pdffitwindow=false,
    pdfstartview={FitH},
    pdftitle={Global bifurcation for a class of nonlinear ODEs},
    pdfauthor={R. G. Bettiol, P. Piccione},
    pdfsubject={},
    pdfkeywords={}
    pdfnewwindow=true,
    colorlinks=true, %set this false for printable version
    linkcolor=blue,
    citecolor=blue,
    urlcolor=black,
}

\newtheorem{theorem}{Theorem}[]
\newtheorem{lemma}[theorem]{Lemma}
\newtheorem{proposition}[theorem]{Proposition}
\newtheorem{corollary}[theorem]{Corollary}

\newtheorem{mainthm}{\bf Theorem}

\theoremstyle{definition}
\newtheorem{definition}[theorem]{Definition}

\theoremstyle{remark}
\newtheorem{remark}[theorem]{Remark}

\newcommand{\dd}{\mathrm d}
\DeclareMathOperator{\vol}{Vol}
\DeclareMathOperator{\Ker}{Ker}
\DeclareMathOperator{\codim}{codim}

\title[Global bifurcation for a class of nonlinear ODE\lowercase{s}]{Global bifurcation\\ for a class of nonlinear ODE\lowercase{s}}


\author[R. G. Bettiol]{Renato G. Bettiol}
\address{City University of New York (Lehman College) \newline
\indent Department of Mathematics  \newline
\indent 250 Bedford Park~Blvd W\newline
\indent Bronx, NY, 10468, USA }
\email{r.bettiol@lehman.cuny.edu}

\author[P. Piccione]{Paolo Piccione}
\address{Universidade de S\~ao Paulo \newline
\indent Departamento de Matem\'atica \newline
\indent Rua do Mat\~ao, 1010 \newline
\indent S\~ao Paulo, SP, 05508-090, Brazil}
\email{piccione@ime.usp.br}

\allowdisplaybreaks
\numberwithin{equation}{section}
\numberwithin{theorem}{section}

\date{\today}

\thanks{The first-named author was supported by grants from the National Science Foundation (DMS-1904342), PSC-CUNY (Award \#62074-00 50), and Fapesp (2019/19891-9). The second named author is partially sponsored by Fapesp (2016/23746-6 and 2019/09045-3) and CNPq, Brazil}

\begin{document}

\begin{abstract}
We briefly survey global bifurcation techniques, and illustrate their use by finding multiple positive periodic  solutions to a class of second order quasilinear ODEs related to the Yamabe problem. As an application, we give a bifurcation-theoretic proof of a classical nonuniqueness result for conformal metrics with constant scalar curvature, that was independently discovered by O.~Kobayashi and R.~Schoen in the 1980s.
\end{abstract}

\maketitle

\begin{section}{Introduction}
Bifurcation Theory and its applications to geometric variational problems has been a long-standing research area of the authors, since the very beginning of their collaboration at the \emph{Instituto de Matem\'atica e Estat\'\i stica} of the \emph{Universidade de S\~ao Paulo}, in Brazil. For that reason, this topic seemed the most natural choice of subject for an article prepared for the special issue of the \emph{S\~ao Paulo Journal of Mathematical Sciences} dedicated to the Golden Jubilee of that Institute. We are very grateful to the editors; in particular to Claudio Gorodski, for giving us the opportunity to contribute to this celebration, which is deeply meaningful to us.
\smallskip

We shall discuss in detail a bifurcation problem for periodic solutions to a class of nonlinear ordinary differential equations related to the Yamabe problem on a closed Riemannian manifold. More precisely, under a suitable Ansatz, the Yamabe equation reduces to an ODE of the type considered in this paper, in which the nonlinearity is given by a power (strictly greater than one) of the unknown function. 

The idea of applying Bifurcation Theory to geometric variational problems is being pursued by 
increasingly many mathematicians, in a growing number of interesting situations.
In particular, our approach here is inspired by a recent work of Betancourt de la Parra, Julio-Batalla, and Petean~\cite{Petean21}, where solutions to the Yamabe equation that are constant along the levelsets of a proper isoparametric function $f\colon M\to [t_0,t_1]$ are studied via a bifurcation problem for an ODE on $[t_0,t_1]$. A key difference, arising from the type of Ansatz considered, is manifested in the ODE boundary conditions: 
while \emph{Neumann} conditions on $[t_0,t_1]$ are needed in~\cite{Petean21}, that geometrically correspond to $f^{-1}(t_0)$ and $f^{-1}(t_1)$ being focal submanifolds (of codimension $\geq2$), 
we deal with \emph{periodic} boundary conditions, which instead correspond to having $f\colon M\to \mathds S^1$, with $\mathds S^1=[t_0,t_1]/\{t_0\sim t_1\}$, and all levelsets $f^{-1}(t)$ being of codimension $1$.
A more elementary introduction to geometric applications of Bifurcation Theory can be found in \cite{bp-notices}.

\subsection*{Main Statements}
We shall present a global bifurcation result for positive periodic solutions (with fixed period) $u=u(t)$ to scalar equations of the form
\begin{equation}\label{eq:bifODE}
u''-\mu (u-\vert u\vert^{q-1}u)=0
\end{equation}
where $\mu>0$ is the bifurcation parameter, and $q>1$ is fixed. Note that, for all $\mu>0$, the unique positive constant solution is $u\equiv1$. Moreover, if $u(t)$ solves \eqref{eq:bifODE}, then any translation $u(t+c)$, $c\in\mathds R$, also solves \eqref{eq:bifODE}. 
We say that two solutions are \emph{distinct} if they are not translations of each other.
The central result is:

\begin{mainthm}\label{mainthm:A}
For all $T>0$, denote by $n(\mu,T)$ the number of distinct positive $T$-periodic solutions to \eqref{eq:bifODE}. Then:
\begin{enumerate}[\rm (a)]
\item $n(\mu,T)=1$, if $\mu\leq\dfrac{4\pi^2}{T^2(q-1)}$;
\item $n(\mu,T) \geq k$, if \, $\dfrac{4\pi^2k^2}{T^2(q-1)} <\mu \leq \dfrac{4\pi^2(k+1)^2}{T^2(q-1)}$, with $k\in\mathds N$.
\end{enumerate}  
In particular, it follows from {\rm (b)} that $\liminf\limits_{\mu\to+\infty} \dfrac{n(\mu,T)}{\sqrt\mu}\geq \dfrac{T}{2\pi}\sqrt{q-1}$.
\end{mainthm}

Statement (a) means that, for $\mu>0$ sufficiently small, the constant function $u\equiv1$ is the unique positive $T$-periodic solution to \eqref{eq:bifODE}, while (b) implies that $\lim\limits_{\mu\to+\infty}n(\mu,T)=+\infty$. 
We show that a positive nonconstant periodic solution to \eqref{eq:bifODE} attains $1$ as a regular value, and give a more precise statement (Theorem~\ref{thm:finalresult}) on the parity and number of zeros of $u-1$, where $u$ is such a solution, yielding the above estimate on $n(\mu,T)$ as $\mu\nearrow+\infty$.\smallskip
 
As an application, we obtain a bifurcation-theoretic proof of a result first obtained (independently) by O.~Kobayashi \cite{Kobayashi85} and R.~Schoen \cite{Schoen89}, on the number of solutions to the Yamabe problem in products $N\times\mathds S^1$, where $N$ is a closed Riemannian manifold with positive constant scalar curvature. This classical result has been extended in several different ways, see e.g.~\cite{petean-2010} and \cite{otoba-peteanDGA}.

\begin{mainthm}\label{mainthmB}
Let $(N^n,g)$, $n\ge2$, be a closed Riemannian manifold with constant scalar curvature $R_N>0$, and $(\mathds S^1,r^2\dd t^2)$ be a circle of length $2\pi r$. 
The number of distinct (unit volume) constant scalar curvature metrics on $M=N\times\mathds S^1$ in the conformal class of the product metric $g\oplus r^2\dd t^2$ goes to infinity as $r\nearrow+\infty$ (at least linearly in $r$).
\end{mainthm}

In the above (see Theorem~\ref{thm:finalresult_yamabe} for details), the distinct constant scalar curvature metrics in the conformal class of the product metric $g\oplus r^2\dd t^2$ are obtained multiplying it by (positive) smooth conformal factors that depend \emph{only} on $t \in\mathds S^1$; i.e., are constant on each slice of the form $N\times\{t\}$, where $t \in\mathds S^1$. 
If $r<\sqrt{n/R_N}$, the (trivial) solution $g\oplus r^2\dd t^2$ is unique among such conformal factors; and, if $(N^n,g)$ is the round sphere $\mathds S^n$, then this uniqueness holds among \emph{all} conformal factors.
\end{section}

\begin{section}{A crash course on classical Bifurcation Theory}\label{sec:crash}
We now provide a brief overview of results from Bifurcation Theory, following the classical approach of the Rabinowitz school, that we hope will serve as an invitation to this beautiful subject. Among the vast literature, we recommend \cite{BufTol2020} for an introduction, and \cite{kielhofer} for a more comprehensive treatise.

Let $X$ and $Y$ be real Banach spaces, $I\subset\mathds R$ be a (possibly infinite) interval, and $\mathcal P\colon I\times X\to Y$ be a function of class $C^\ell$, $\ell \ge1$, satisfying:
\begin{equation}\label{eq:basbifeqn}
\phantom{\qquad\forall\,\mu\in I.}\mathcal P(\mu,x_0)=0,\qquad\forall\,\mu\in I,
\end{equation}
for some fixed $x_0\in X$. Let us denote by $\mathcal S$ the set of zeros of $\mathcal P$, that is,
\begin{equation}\label{eq:defG}
\mathcal S:=\mathcal P^{-1}(0)\subset I\times X.
\end{equation}
For $\mu_0\in I$, the point $(\mu_0,x_0)$ is a \emph{bifurcation point} for the equation \eqref{eq:basbifeqn} if every neighborhood of $(\mu_0,x_0)$ in $I\times X$ contains some point $(\mu,x)$, with $x\ne x_0$, such that $\mathcal P(\mu,x)=0$. In the usual terminology, the set $\mathcal B_{\mathrm{triv}}=I\times\{x_0\}\subset I\times X$ is referred to as the \emph{trivial branch of solutions} for the equation $\mathcal P(\mu,x)=0$, and if $(\mu_0,x_0)$ is a bifurcation point for the equation \eqref{eq:basbifeqn}, then $\mu_0$ is an \emph{instant of bifurcation} along $\mathcal B_{\mathrm{triv}}$. Alternatively, $\mu_0$ is an instant of bifurcation along $\mathcal B_{\mathrm{triv}}$ if $(\mu_0,x_0)$ belongs to the closure in $I\times X$ of the set 
$\mathcal S\setminus\mathcal B_{\mathrm{triv}}$. The connected component of $(\mu_0,x_0)$ in the closure of 
$\mathcal S\setminus\mathcal B_{\mathrm{triv}}$ is called the \emph{bifurcation branch} issuing from the bifurcation point $(\mu_0,x_0)$.

An immediate application of the Implicit Function Theorem gives a necessary condition for the existence of bifurcation branches issuing from a point $(\mu_0,x_0)$ along the trivial branch for equation \eqref{eq:basbifeqn}: if the derivative 
$\frac{\partial\mathcal P}{\partial x}(\mu_0,x_0)\colon X\to Y$
is an isomorphism, then bifurcation cannot occur at $(\mu_0,x_0)$. It is also well-known that failure of this condition is, in general, \emph{not sufficient} to guarantee bifurcation. 
What goes under the name of \emph{Bifurcation Theory} is a collection of results giving sufficient conditions for the existence of bifurcation, and describing the geometry of the bifurcation set. Typically, existence of bifurcation is proven by topological methods; more precisely, bifurcation is detected by a jump of some topological invariant associated to the solutions along the trivial branch. As an example, if \eqref{eq:basbifeqn} is of a variational nature, i.e., if $\mathcal P(\mu,\cdot)$ is the derivative of some real-valued $C^2$-function $f_\mu$ having $x_0$ as a critical point for all $\mu$, then a sufficient condition for bifurcation at $(\mu_0,x_0)$ is that the \emph{Morse index} of $f_\mu$ at $x_0$ jumps at $\mu=\mu_0$. 

Standard topological methods are better suited to finite-dimensional problems, despite the fact that the most interesting bifurcation problems for ODEs, PDEs, etc., involve an infinite-dimensional setup. It turns out that a finite-dimensional reduction in bifurcation problems, as well as in many other nonlinear functional analytical problems, is possible under suitable \emph{Fredholmness} assumptions. The most classical result in this direction is the so-called \emph{Lyapunov--Schmidt reduction}. The central result of this theory gives a description of the set $\mathcal S$ near a point $(\mu_0,x_*)\in \mathcal S\subset  I\times X$ in terms of the zero set of a smooth function on a finite-dimensional space, under a Fredholmness assumption:

\begin{theorem}[Lyapunov--Schmidt reduction]\label{thm:LSreduction}
Suppose that $\mathcal P(\mu_0,x_*)=0$, with $(\mu_0,x_*)\in I\times X$, and assume that the derivative
$L=\frac{\partial\mathcal P}{\partial x}(\mu_0,x_*)$,
\[L\colon X\longrightarrow Y,\]
is a Fredholm operator, with $\Ker(L)\ne\{0\}$ and $k=\codim_Y\!\big(L(X)\big)$. Then, there exist open sets $U\subset X$, $V\subset\mathds R\times\Ker(L)$, and maps of class $C^\ell$
\[\psi\colon V\to X,\qquad h\colon V\to\mathds R^k,\]
with $(\mu_0,x_*)\in U$, $(\mu_0,0)\in V$, and $\psi(\mu_0,0)=x_*$, such that
\[(\mu,x)\in\mathcal S\cap U\quad\Longleftrightarrow\quad\psi(\mu,\xi)=x\ \ \text{for some}\ (\mu,\xi)\in h^{-1}(0).\]
\end{theorem}
\begin{proof}
See e.g.~\cite[Theorem~8.2.1, p.\ 126]{BufTol2020}, or \cite[Theorem~I.2.3, p.\ 7]{kielhofer}.
\end{proof}
In other words, Theorem~\ref{thm:LSreduction} reduces the equation $\mathcal P(\mu,x)=0$, for $(\mu,x)$ near $(\mu_0,x_*)$, to the equation $h(\mu,\xi)=0$, where $h$ is a function whose domain is an open subset of a finite-dimensional space.

The Lyapunov--Schmidt reduction and the Implicit Function Theorem are the essential ingredients for the proof of a classical bifurcation result of Crandall and Rabinowitz \cite{crandall-rabinowitz}, known in the literature as \emph{bifurcation from simple eigenvalues}. In order to give a basic statement of this result, let us go back to considering the trivial branch of solutions $x=x_0$ to the equation $\mathcal P(\mu,x)=0$.

\begin{theorem}[Crandall--Rabinowitz]\label{thm:CraRab}
Using the same notation as in Theorem~\ref{thm:LSreduction} with $x_*=x_0$, assume $\mathcal P$ is a map of class $C^\ell$ with $\ell\in\{2,\cdots,\infty,\omega\}$, and that:
\begin{enumerate}[\rm (a)]
\item  $L=\frac{\partial\mathcal P}{\partial x}(\mu_0,x_0)$ is a Fredholm operator of index $0$;
\item  $\Ker(L)=\mathds R\cdot \xi_0$ for some $\xi_0\in X\setminus\{0\}$, i.e., $\Ker(L)$ is one-dimensional;
\item  $\dfrac{\dd}{\dd\mu}\Big\vert_{\mu=\mu_0}\left(\dfrac{\partial\mathcal P}{\partial x}(\mu,x_0)\xi_0\right) \not\in L(X)$.
\end{enumerate}
Then $(\mu_0,x_0)$ is a bifurcation point for the equation \eqref{eq:basbifeqn}. More precisely, for a sufficiently small neighborhood $U$ of $(\mu_0,x_0)$ in $I\times X$, the set $\mathcal S\cap U$ consists of the points of the trivial branch $U\cap\mathcal B_{\mathrm{triv}}$, and the points in the support of a $C^{\ell-1}$-path $\left(-\varepsilon,\varepsilon\right)\ni s\mapsto\big(\mu(s),x(s)\big)\in I\times X$, with $\mu(0)=\mu_0$, 
$x(0)=x_0$, $x(s)\ne x_0$ for $s\ne0$,
and $x'(0)=\xi_0$.
\end{theorem}

\begin{proof}
The original result is proven in \cite[Theorem~1.7]{crandall-rabinowitz}, see also \cite[Theorem~I.4.1]{kielhofer} or \cite[Theorem~8.3.1]{BufTol2020}.
\end{proof}

The next basic result from Bifurcation Theory deals with the global geometry of a bifurcation branch. The following is a theorem originally due to Rabinowitz \cite{rabinowitz}, which states that, under a suitable properness assumption, the bifurcation branch of solutions to $\mathcal P(\mu,x)=0$ issuing from a bifurcation point $(\mu_0,x_0)$, with the assumptions of Theorem~\ref{thm:CraRab}, satisfies the following dichotomy: it either reattaches to the trivial branch, or it is noncompact.\smallskip

\begin{theorem}[Rabinowitz]\label{thm:rabinowitz}
Let $(\mu_0,x_0)$ be a bifurcation point for the equation $\mathcal P(\mu,x)=0$ that satisfies the assumptions of Theorem~\ref{thm:CraRab}, and let $\mathcal B_{\mu_0}$ be the bifurcation branch issuing from $(\mu_0,x_0)$. Denote by $(-\varepsilon,\varepsilon)\ni s\mapsto\big(\mu(s),x(s)\big)$ the local bifurcation branch of solutions to $\mathcal P(\mu,x)=0$ near $(\mu_0,x_0)$, and assume that $s\mapsto \mu(s)$ is not constant near $s=0$.
Assume also that the restriction to the closed and bounded subsets of $\mathcal S$ of 
the projection $\Pi\colon I\times X\to I$ onto the first factor is a proper map. 
Then, the path $s\mapsto\big(\mu(s),x(s)\big)$ can be extended to a continuous path
$\left(-\varepsilon,+\infty\right)\ni s\mapsto \big(\mu(s),x(s)\big)\in I\times X$ whose support is contained in $\mathcal B_{\mu_0}$, such that either one of the two alternatives occurs:
\begin{enumerate}[\rm (A)]
\item $\lim\limits_{s\to+\infty}\big(\mu(s),x(s)\big)=(\mu_1,x_0)$, $\mu_1\ne\mu_0$;
\item $\big(\mu(s),x(s)\big)$ approaches the boundary of $I\times X$ as $s\to+\infty$, i.e., one of the following is true:
\begin{enumerate}[\rm (B-1)]
\item $\lim\limits_{s\to+\infty}\big(\mu(s),x(s)\big)\in\partial I\times X$; 
\item $\lim\limits_{s\to+\infty}\big\Vert x(s)\Vert=+\infty$.
\end{enumerate}
\end{enumerate}
If $\mathcal P$ is real analytic, then $s\mapsto\big(\mu(s),x(s)\big)$ can also be chosen real analytic.
\end{theorem}

\begin{proof}
See for instance \cite[Theorem~1.3]{rabinowitz} or \cite[Theorem II.5.8]{kielhofer}, or \cite[Theorem~9.1.1]{BufTol2020} for the real analytic case.
\end{proof}

\begin{remark}\label{rem:other-side}
The path $s\mapsto \big(\mu(s),x(s)\big)$ in \Cref{thm:rabinowitz} can be further extended to $\left(-\infty,+\infty\right)\ni s\mapsto \big(\mu(s),x(s)\big)\in I\times X$, by applying the result again replacing $s$ with $-s$. Alternatives (A) and (B) are independent at each of the limits $s=\pm \infty$.    
\end{remark}

The above Theorems~\ref{thm:CraRab} and~\ref{thm:rabinowitz} are the main tools underlying the proof of the results in the present paper, as well as several other nonuniqueness results in geometric problems studied by the authors, see e.g.~\cite{bp-notices} for a short survey.
\end{section}

\begin{section}{A compactness result}
In this section, we consider more general second-order periodic boundary value problems of the form \eqref{eq:bvp2}, before specializing to the quasilinear equation \eqref{eq:bifODE}.

\begin{proposition}\label{prop:compactness}
Assume that $f\colon\mathds R\to\mathds R$ is a function of class $C^\ell$, $\ell\ge1$, that satisfies $f(0)=f(u_0)=0$ for some $u_0>0$, 
$f<0$ in $\left(0,u_0\right)$ and $f>0$ in $\left(u_0,+\infty\right)$. Then, the set of positive solutions to the boundary value problem:
\begin{equation}\label{eq:bvp2}
u''+f(u)=0,\quad u(a)=u(b),\quad u'(a)=u'(b),
\end{equation}
is compact in the $C^{\ell+1}$-topology. 
\end{proposition}

\begin{proof}
The boundary conditions (and the fact that the equation is autonomous) imply that every solution to \eqref{eq:bvp2} can be extended to a periodic function on $\mathds R$ with period $b-a$; such an extension satisfies the differential equation in \eqref{eq:bvp2} on $\mathds R$.
The boundary value problem has exactly two constant nonnegative solutions, given by $u\equiv0$ and $u\equiv u_0$. Every other positive solution $u$ of \eqref{eq:bvp2} must satisfy:
\begin{equation}\label{eq:passthroughu0}
0<\min_{[a,b]} u< u_0<\max_{[a,b]}u.
\end{equation}

Namely, by periodicity, every such solution $u$ must have a minimum and a maximum point in $[a,b]$, where $u'=0$. Neither the maximum nor the minimum of $u$ can be equal to $u_0$, for otherwise $u\equiv u_0$. Now, the minimum must occur at some point where $u''=-f(u)\ge0$, and this implies that the minimum of $u$ is in $\left(0,u_0\right)$. Similarly, the maximum of $u$ must occur at some point where $u''=-f(u)\le0$, and therefore the maximum of $u$ is greater than $u_0$.

We now claim that there exists a positive constant $t_0>0$ such that every positive solution to \eqref{eq:bvp2} has minimum which is greater than or equal to $t_0$. Namely, if this were not the case, there would exist a sequence $(u_n)_n$ of positive solutions to \eqref{eq:bvp2}, and a sequence $(x_n)_n$ in $[a,b]$ such that $u_n'(x_n)=0$ and $\lim\limits_{n\to\infty}u_n(x_n)=0$. By the continuous dependence of the solution to an ODE on the initial data, this would imply that $u_n$ tends to $0$ uniformly on $[a,b]$, which contradicts \eqref{eq:passthroughu0}. This proves our claim about the existence of $t_0$.
 
From this, we first deduce that the set of positive solutions to \eqref{eq:bvp2}, which is contained in $C^{\ell+1}\big([a,b],\mathds R\big)$, is closed in the $C^1$-topology.
Using again the continuous dependence on the initial data, if $(u_n)_n$ is a sequence of positive solutions to \eqref{eq:bvp2}, and $(x_n)_n$ is a sequence in $[a,b]$ such that $\min_{[a,b]}u_n=u_n(x_n)$ for all $n$, then (up to subsequences) we may assume that $\lim_{n\to\infty}x_n=x_*\in[a,b]$ and that $\lim_{n\to\infty}u_n(x_n)=t_*\in[t_0,u_0]$, and therefore $u_n$ is $C^{\ell+1}$-convergent to the solution to the initial value problem
\[u''+f(u)=0,\quad u(x_*)=t_*,\quad u'(x_*)=0.\]
This solution must be defined on all of $[a,b]$; and, since the set of positive solutions to \eqref{eq:bvp2} is closed, it follows that $u$ is a positive solution to \eqref{eq:bvp2}.
\end{proof}

The proof of Proposition~\ref{prop:compactness} reveals the following regarding solutions to \eqref{eq:bvp2}:

\begin{corollary}\label{thm:qualitative}
Under the assumptions of Proposition~\ref{prop:compactness}, we have that:
\begin{itemize}
\item the only constant solutions to \eqref{eq:bvp2} are $u\equiv0$ and $u\equiv u_0$;
\item given any positive nonconstant periodic solution $u$ of \eqref{eq:bvp2}, the function $u-u_0$ has at least one zero in $[a,b]$, and all its zeros are simple.
\end{itemize}
\end{corollary}

\begin{proof}
The first statement is trivial, since $f$ has exactly two zeros. As to the second statement, from \eqref{eq:passthroughu0} it follows that, given any positive nonconstant solution $u$ to \eqref{eq:bvp2}, the function $u-u_0$ must vanish somewhere in $[a,b]$. Any zero $t_0\in[a,b]$ of $u-u_0$ must be simple, for otherwise, by the uniqueness of the solution to the initial value problem $u''+f(u)=0$, $u(t_0)=u_0$, $u'(t_0)=0$, it would be $u\equiv u_0$.
\end{proof}
\end{section}

\begin{section}{Regularity and Positivity}
Recall we are interested in positive periodic solutions to \eqref{eq:bifODE}; more specifically, in branches of periodic solutions that issue from the positive constant solution.
Our first result is to show that, along one such bifurcation branch, all solutions remain positive. 
In order to give a precise statement, consider the sets
\begin{equation*}
\mathcal E:=\Big\{(\mu,u)\in \left(0,+\infty\right)\times C^2\big([a,b],\mathds R\big):u''-\mu\big(u-\vert u\vert^{q-1}u)=0\Big\},
\end{equation*} 
and
\begin{equation*}
\mathcal E':=\big\{(\mu,u)\in\mathcal E:u'(a)=u'(b)\big\}.
\end{equation*}
Clearly, $\mathcal E$ is $C^0$-closed and $\mathcal E'$ is $C^1$-closed in $\left(0,+\infty\right)\times C^2\big([a,b],\mathds R\big)$.

\begin{proposition}\label{thm:positivityinabranch}
Let $\mathcal C\subset\mathcal E'$ be a connected and closed subset of $\mathcal E'$ relatively to the $C^0$-topology, and assume the existence of $(\mu_*,u_*)\in\mathcal C$ such that $u_*>0$ on $[a,b]$. Then, for all $(\mu,u)\in\mathcal C$, $u>0$ on $[a,b]$. 
\end{proposition} 

\begin{proof}
Set $\widetilde{\mathcal C}=\big\{(\mu,u)\in\mathcal E':u>0\ \text{on}\ [a,b]\big\}$. Clearly, $\widetilde{\mathcal C}$ is $C^0$-open in $\mathcal E'$, and $\widetilde{\mathcal C}\ne\emptyset$. In order to show that $\widetilde{\mathcal C}=\mathcal C$, it suffices to show $\widetilde{\mathcal C}$ is $C^0$-closed in $\mathcal E'$. Assume $(\mu_n,u_n)_n$ is a sequence in $\widetilde{\mathcal C}$ which $C^0$-converges to $(\mu_\infty,u_\infty)\in\mathcal C$. Then, $u_\infty\ge0$ on $[a,b]$ and if $u_\infty$ vanishes at some $t_0\in[a,b]$, then also $u'_\infty(t_0)=0$, and so, by uniqueness, $u_\infty\equiv0$. Thus, we have two possibilities: either $u_\infty>0$ on $[a,b]$, or $u_\infty\equiv0$; let us show that this second case does not occur. If $u_n$ is uniformly close to $0$, and positive, since $\mu_n$ is close to $\mu_\infty>0$, then $u_n''=\mu_n(u_n-u_n^q)<0$ on $[a,b]$ for $n$ sufficiently large. This implies that $u_n'$ is strictly decreasing on $[a,b]$, and so $u'(a)>u'(b)$, which gives a contradiction. Hence $u_\infty>0$, and this proves that $\widetilde{\mathcal C}$ is closed in $\mathcal C$, concluding the proof.
\end{proof}
The result of Proposition~\ref{thm:positivityinabranch} will be invoked in the proof of Proposition~\ref{thm:globbif}, when it will be necessary to show that periodic solutions to \eqref{eq:bifODE} that bifurcate from the the trivial (positive) solution $u\equiv1$ remain positive.\smallskip
 
Finally, let  us observe that every $C^2$-solution to \eqref{eq:bifODE} is $C^\infty$, and that all the $C^\ell$-topologies coincide in the space of periodic solutions to \eqref{eq:bifODE}. We will implicitly assume throughout that the set of solutions to \eqref{eq:bifODE} is endowed with such topology.
\end{section}

\section{A uniqueness result}\label{sec:uniqueness}
Let us now consider the ODE
\begin{equation}\label{eq:slightlymoregeneral}
u''=\mu(u-u^q),
\end{equation}
with $\mu>0$ and $q>1$, which fits the setup of Proposition~\ref{prop:compactness} and Corollary~\ref{thm:qualitative}. We know from Proposition~\ref{prop:compactness} that, for all $T>0$, the set of all $T$-periodic positive solutions to \eqref{eq:slightlymoregeneral} is compact. Closely following the ingenious  arguments of Licois and V\'eron~\cite[Sec.~2]{LicVer98}, we now prove that, if $\mu>0$ is small enough, then the constant function $u\equiv1$ is the unique positive $T$-periodic solution to \eqref{eq:slightlymoregeneral}.

Without loss of generality, we will consider the case $T=2\pi$.
Namely, if $u$ is a $T$-periodic solution to \eqref{eq:slightlymoregeneral}, then $\overline u(t)=u\big(\frac T{2\pi}t\big)$ is a $2\pi$-periodic solution to
\[\overline u''=\tfrac{T^2}{4\pi^2}\mu(\overline u-\overline u^q).\]
Furthermore, note that the map $u\mapsto\mu^\frac1{q-1}\cdot u$ gives a bijection from the set of solutions to \eqref{eq:slightlymoregeneral} to the set of solutions to the equation
\begin{equation}\label{eq:slightlymoregeneral1varepsilon}
u''-\mu\,u+u^q=0,
\end{equation}
so we may (and will) work with \eqref{eq:slightlymoregeneral1varepsilon} instead of \eqref{eq:slightlymoregeneral} to prove the uniqueness~result.

Let us assume that $u$ is a positive $2\pi$-periodic solution to \eqref{eq:slightlymoregeneral1varepsilon}; the periodicity assumption will be used to eliminate all the boundary terms in the integrations by part in the computations below. Let us set
\begin{equation*}
v=u^{-1},
\end{equation*}
and observe that the corresponding equation satisfied by $v$ is:
\begin{equation}\label{eq:2.8}
-v''+2\frac{(v')^2}v+v^{2-q}-\mu\,v=0.
\end{equation}
Multiplying \eqref{eq:2.8} by $v^{-4}(v')^2$, and integrating on $[0,2\pi]$, we get:
\begin{multline}\label{eq:2.16}
\int_0^{2\pi} v''\,v^{-4}(v')^2\,\mathrm dt=\int_0^{2\pi}\big[2v^{-5}(v')^4+(v')^2v^{-2-q}-\mu\,(v')^2\,v^{-3}\big]\,\mathrm dt.
\end{multline}
We now multiply \eqref{eq:2.8} by $v^{-3} v''$ and integrate on $[0,2\pi]$, obtaining:
\begin{align}\nonumber
\int_0^{2\pi}v^{-3}(v'')^2\,\mathrm dt&=\int_0^{2\pi}\Big[2v^{-4}v''(v')^2 +v''\big(v^{-1-q}-\mu\,v^{-2}\big)\Big]\,\mathrm dt\\ \label{eq:2.17}
&=\int_0^{2\pi}\Big[2v^{-4}v''(v')^2-v'\left(-(q+1)v'v^{-q-2}+2\mu v'v^{-3}\right)\Big]
\,\mathrm dt\\\nonumber
&=\int_0^{2\pi}\Big[2v^{-4}v''(v')^2+(q+1)(v')^2v^{-q-2}-2\mu(v')^2v^{-3}\Big]
\,\mathrm dt.
\end{align}
We will now eliminate the term $\int_0^{2\pi}v^{-q-2}(v')^2\,\mathrm dt$
from \eqref{eq:2.16} and \eqref{eq:2.17}, obtaining:
\begin{multline}\label{eq:2.18}
\mu(q-1)\int_0^{2\pi}v^{-3}(v')^2\,\mathrm dt=\\
-(q+3)\int_0^{2\pi}v^{-4}v''(v')^2\,\mathrm dt+\int_0^{2\pi}v^{-3}(v'')^2\,\mathrm dt+
2(q+1)\int_0^{2\pi}v^{-5}(v')^4\,\mathrm dt.
\end{multline}
An immediate calculation gives
\[\left[\big(v^{-\frac{1}2}\big)''\right]^2=\frac{9}{16}v^{-5}(v')^4+\frac14(v'')^2v^{-3}-\frac34v^{-4}(v')^2v'',\]
so
\begin{equation}\label{eq:2.20}
v^{-4}(v')^2v''=-\frac43\left[\big(v^{-\frac{1}2}\big)''\right]^2+\frac34v^{-5}(v')^4+\frac13 v^{-3}(v'')^2.
\end{equation}
Substituting \eqref{eq:2.20} in \eqref{eq:2.18} gives:
\begin{multline}\label{eq:2.22}
\mu(q-1)\int_0^{2\pi}v^{-3}(v')^2\,\mathrm dt=\\-\frac q3\int_0^{2\pi}v^{-3}(v'')^2\,\mathrm dt+\frac{4}{3}(q+3)\int_0^{2\pi}\left[\big(v^{-\frac{1}2}\big)''\right]^2\,\mathrm dt+\frac14(5q-1)\int_0^{2\pi}v^{-5}(v')^4\,\mathrm dt.
\end{multline}
Integration by parts easily yields
\begin{equation}\label{eq:2.19}
3\int_0^{2\pi}v^{-4}(v')^2v''\,\mathrm dt=4\int_0^{2\pi}v^{-5}(v')^4\,\mathrm dt,
\end{equation}
and, substituting \eqref{eq:2.20} in \eqref{eq:2.19} gives:
\begin{equation}\label{eq:2.23}
\int_0^{2\pi}v^{-3}(v'')^2\,\mathrm dt=
\frac{7}{4}\int_0^{2\pi}v^{-5}(v')^4\,\mathrm dt+4\int_0^{2\pi}\left[\big(v^{-\frac12}\big)''\right]^2\,\mathrm dt.
\end{equation}
Finally, substituting \eqref{eq:2.23} into \eqref{eq:2.22} gives (cf.~\cite[(2.9)]{LicVer98}):
\begin{multline}\label{eq:2.9}
\mu(q-1)\int_0^{2\pi}v^{-3}(v')^2\,\mathrm dt=\\
\left(\frac23q-\frac14\right)\int_0^{2\pi}v^{-5}(v')^4\,\mathrm dt+4\int_0^{2\pi}\left[\big(v^{-\frac12}\big)''\right]^2\,\mathrm dt.
\end{multline}
Now, if $f$ is a $2\pi$-periodic function of class $C^2$, then $f'$ is $L^2$-orthogonal to the space of constant functions on $[0,2\pi]$, so Wirtinger's inequality yields:
\begin{equation}\label{eq:poincare}
\int_0^{2\pi}\big\vert f'(t)\big\vert^2\,\mathrm dt\le\int_0^{2\pi}\big\vert f''(t)\big\vert^2\,\mathrm dt.
\end{equation}
Applying \eqref{eq:poincare} to $f=v^{-\frac12}$, we obtain:
\begin{equation}\label{eq:2.24}
4\int_0^{2\pi}\left[\big(v^{-\frac12}\big)''\right]^2\,\mathrm dt\ge \int_0^{2\pi}v^{-3}(v')^2\,\mathrm dt.
\end{equation}
We are now ready to prove the claimed uniqueness result:

\begin{theorem}\label{thm:uniqepsismall}
For all $q>1$, if $\mu\in\left(0,\frac1{q-1}\right]$, then equation \eqref{eq:slightlymoregeneral} has a unique positive $2\pi$-periodic solution, given by the constant function $u\equiv1$.
\end{theorem}

\begin{proof}
As explained in the beginning of the section, the desired conclusion is equivalent to showing that $u\equiv\mu^{\frac1{q-1}}$ is the unique positive $2\pi$-periodic solution to equation \eqref{eq:slightlymoregeneral1varepsilon}.
Assume that $\mu\le\frac1{q-1}$ and that $u$ is a positive $2\pi$-periodic solution to \eqref{eq:slightlymoregeneral1varepsilon}. If $u$ is nonconstant, then all three integrals appearing in \eqref{eq:2.9} are positive. Thus, from $\mu(q-1)\leq 1$, \eqref{eq:2.9}  and \eqref{eq:2.24}, we obtain:
\begin{multline*}\int_0^{2\pi}v^{-3}(v')^2\,\mathrm dt \ge\mu(q-1)\int_0^{2\pi}v^{-3}(v')^2\,\mathrm dt\\ \stackrel{\eqref{eq:2.9}}>4\int_0^{2\pi}\left[\big(v^{-\frac{1}2}\big)''\right]^2\,\mathrm dt\stackrel{\eqref{eq:2.24}}\ge
\int_0^{2\pi}v^{-3}(v')^2\,\mathrm dt,
\end{multline*}
which gives a contradiction, completing the proof.
\end{proof}

\begin{remark}
Theorem~\ref{thm:uniqepsismall} is sharp; as we will see, equation \eqref{eq:slightlymoregeneral} has multiple positive $2\pi$-periodic solutions if $\mu>\frac1{q-1}$.
\end{remark}

\begin{section}{Bifurcation}
\label{sec:globalbif}

In this section, we prove \Cref{mainthm:A} in the Introduction, by establishing a global bifurcation result for positive periodic solutions to \eqref{eq:bifODE}. Fix $q>1$, $T>0$, and, as we are dealing with positive solutions to \eqref{eq:bifODE}, consider the nonlinear problem:
\begin{equation}\label{eq:problemOlambdaT}
\mathbf P_{\mu}:\qquad \left\{\begin{aligned}&u''-\mu \, u+\mu\,  u^q=0,\\
&u(0)=u(T),\quad u'(0)=u'(T),\\
& u>0\ \text{ on }\ [0,T],
\end{aligned}\right.\phantom{\mathbf P_{\mu}:\qquad }
\end{equation} 
where $\mu$ is a positive real parameter. 
Note that the constant function
\[u_0\equiv1\]
solves $\mathbf P_{\mu}$ for all $\mu$, and that every positive solution to $\mathbf P_{\mu}$ admits a $T$-periodic smooth extension to $\mathds R$.

\subsection{A priori bounds}
We know by Proposition~\ref{prop:compactness} that, for all $\mu>0$ fixed, the solutions to $\mathbf P_{\mu}$ form a compact subset of $C^\ell\big([0,T],\mathds R\big)$ for all $\ell\ge0$. However, in order to apply global bifurcation results, we need a compactness property that is \emph{locally uniform with respect to $\mu$}. This could be obtained with a refinement of the proof of Proposition~\ref{prop:compactness}, but we present a different argument (using 
phase portraits) which is specific to problem $\mathbf P_{\mu}$  and inspired by observations of Schoen~\cite{Schoen89}.

Looking closely at the ODE in problem $\mathbf P_{\mu}$, which is the same as \eqref{eq:slightlymoregeneral}, namely
\begin{equation}\label{eq:eqOlambdaT}
u''-\mu\,u+\mu\,u^q=0,
\end{equation}
one realizes that $u(t)=A\cdot\cosh(B\,t)^C$ is a (positive) solution for a suitable choice of coefficients $A>1$, $B>0$, and $C<0$. Namely, such $u(t)$ solves \eqref{eq:eqOlambdaT} setting
\[\textstyle  A_q=\left(\tfrac{q+1}2\right)^\frac1{q-1},\quad B_{q,\mu}=\frac{q-1}2\,\sqrt\mu,\quad C_q=-\frac2{q-1}.\]
We denote this solution by
\begin{equation}\label{eq:upepsilon}
\phantom{,\quad t\in\mathds R.}
u_{q,\mu}(t):=A_q\,\cosh(B_{q,\mu}\,t)^{C_q},\quad t\in\mathds R,
\end{equation}
and observe that $1<A_q<\sqrt{e}$ is uniformly bounded for all $q>1$. 
The path $\mathds R\ni t\mapsto\big(u_{q,\mu}(t),u'_{q,\mu}(t)\big)\in\mathds R^2$ in phase space, together with the point $(0,0)$, which belongs to its closure, bounds a compact subset $\Omega_{q,\mu}$ of $\mathds R^2$ that contains the point $(1,0)$ and is invariant by the flow of \eqref{eq:eqOlambdaT}, see Figure~\ref{fig:omegaqepsilon}. 
\begin{figure}[htbp]
\vspace{-.3cm}
\centering
\includegraphics[width=0.8\columnwidth]{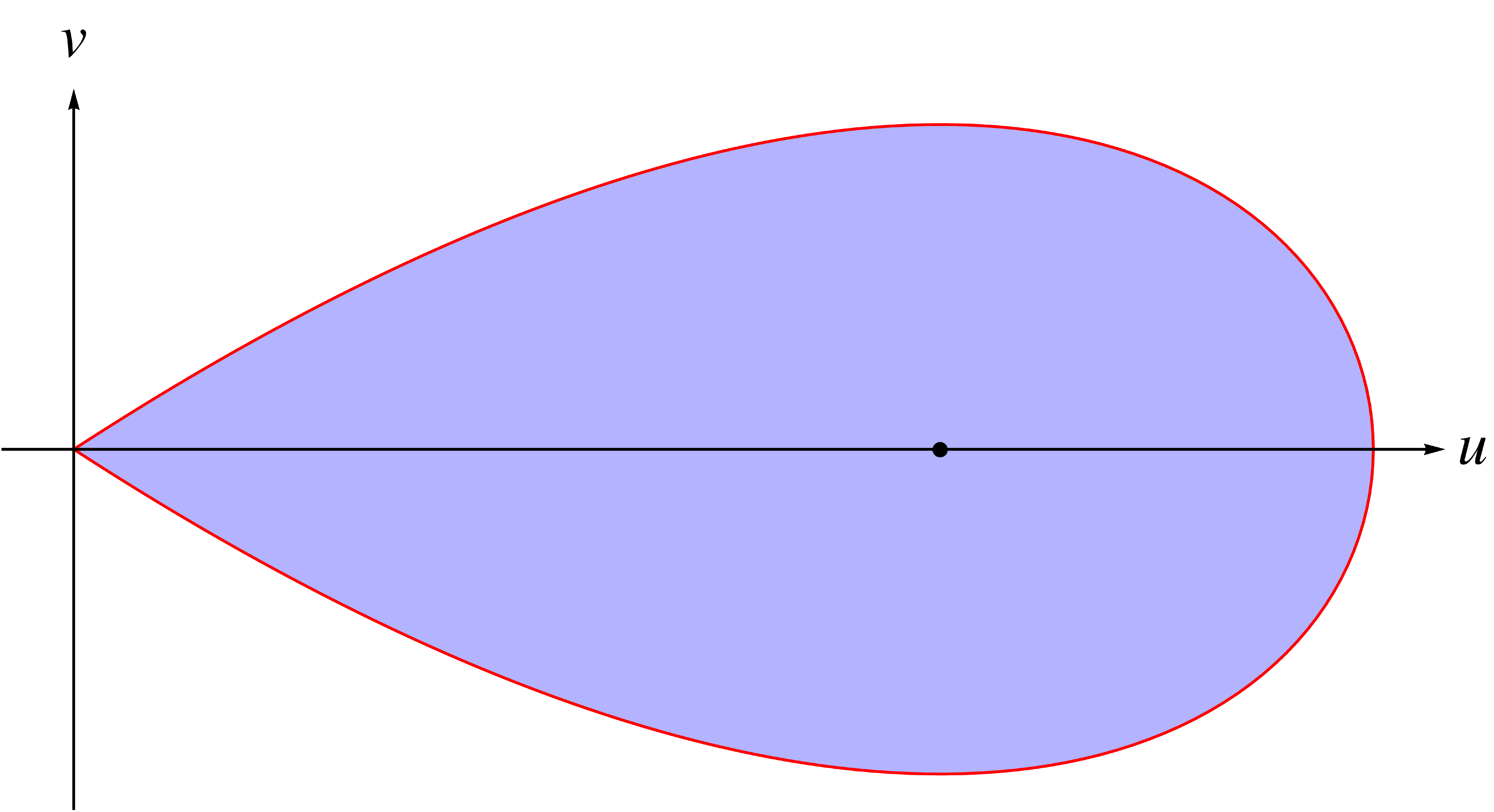}
\caption{The region $\Omega_{q,\mu}$, containing the point $(1,0)$.} 
\label{fig:omegaqepsilon}
\vspace{-.3cm}
\end{figure} 

\begin{proposition}\label{thm:uniformbounds}
For all $q>1$ and all $\mu>0$, every positive periodic solution to \eqref{eq:eqOlambdaT} defines a path in phase space which is contained in the interior of $\Omega_{q,\mu}$. In particular, the $C^0$-norm of every positive periodic solution to \eqref{eq:eqOlambdaT} admits a uniform bound independent of $\mu$.
\end{proposition}

\begin{proof}
Trajectories in phase space of periodic solutions to \eqref{eq:eqOlambdaT} correspond to closed orbits of the planar vector field 
\[\vec F(u,v)=\big(v,\mu (u-u^q)\big).\] 
By the Poincar\'e--Bendixson Theorem, every closed orbit of $\vec F$ bounds a (compact) set which is invariant by the flow of $\vec F$, and which contains a fixed point of that flow, i.e., a zero of $\vec F$. 
Clearly, in the case of a positive periodic solution to \eqref{eq:eqOlambdaT}, such zero must be $(u_0,v_0)=(1,0)$. Every closed simple curve that bounds an open domain containing $(u_0,v_0)$ must intersect the interior of $\Omega_{q,\mu}$, and therefore it must be entirely contained in the interior of $\Omega_{q,\mu}$. This proves the first statement of the Proposition. For the second, note that the projection $\mathds R^2\ni(u,v)\mapsto u\in\mathds R$ maps $\Omega_{q,\mu}$ to the compact interval $\left[0,A_q\right]$, so that every positive periodic solution $u$ of \eqref{eq:eqOlambdaT} satisfies $\Vert u\Vert_\infty\le A_q$.
\end{proof}

We are now ready to state and prove the uniform compactness property needed to apply the global bifurcation result (Theorem~\ref{thm:rabinowitz}) to problem $\mathbf P_{\mu}$.

\begin{corollary}\label{thm:properprojection}
The restriction  to the subset $\big\{(\mu,u):u\ \text{is a solution to }\mathbf P_{\mu}\big\}$  of the projection $\Pi\colon \left(0,+\infty\right)\times C^\ell\big([0,T],\mathds R\big) \to \left(0,+\infty\right)$ is a proper map.
\end{corollary}

\begin{proof}
By Proposition~\ref{thm:uniformbounds}, the set ${\Pi}^{-1}\big(\left(0,+\infty\right)\big)$ is bounded in $L^\infty$. From~\eqref{eq:eqOlambdaT}, for any pair $0<\mu_a\le\mu_b$, the set ${\Pi}^{-1}\big(\left[\mu_a,\mu_b\right]\big)$ is bounded in $C^{\ell+2}\big([0,T],\mathds R)$. Such a set is clearly closed in the $C^\ell$-topology, for all $\ell\ge0$. By the Arzel\'a--Ascoli Theorem, ${\Pi}^{-1}\big(\left[\mu_a,\mu_b\right]\big)$ is compact in $\left(0,+\infty\right)\times C^\ell\big([0,T],\mathds R\big)$.
\end{proof}

\subsection{Parity}
Solutions to $\mathbf P_{\mu}$ are the positive solutions to \eqref{eq:bifODE} in the Banach space
\begin{equation}\label{eq:C2per}
C^2_{\mathrm{per}}\big([0,T],\mathds R\big):=\Big\{u\in C^2\big([0,T],\mathds R\big):u(0)=u(T),\ u'(0)=u'(T)\Big\}.
\end{equation}
Let us observe that any (positive) solution to \eqref{eq:bifODE} in $C^2_{\mathrm{per}}\big([0,T],\mathds R\big)$ admits a smooth, i.e., $C^\infty$, $T$-periodic extension to $\mathds R$. Such extension clearly solves \eqref{eq:bifODE} on the entire real line. With a slight abuse of notation, we will sometimes identify solutions to 
\eqref{eq:bifODE} in $C^2_{\mathrm{per}}\big([0,T],\mathds R\big)$ with their $T$-periodic extension to $\mathds R$.

Define a linear isometry $\mathcal I$ of \eqref{eq:C2per} by setting $\mathcal I(u):=u^-$, where:
\[\phantom{\qquad\forall\,t\in[a,b].}u^-(t)=u(T-t),\qquad\forall\,t\in[0,T].\]
Since \eqref{eq:eqOlambdaT} does not involve the first-order term $u'$, 
it is easy to see that the set of (positive) solutions to \eqref{eq:eqOlambdaT} is invariant under $\mathcal I$.
We are then led to consider the closed subspace of \eqref{eq:C2per} fixed by $\mathcal I$, which consists of functions in \eqref{eq:C2per} that are \emph{even} about the midpoint $T/2$ of the interval $[0,T]$. We denote this space by:
\begin{equation}\label{eq:C2pereven}
C^2_{\mathrm{per}}\big([0,T],\mathds R\big)^\mathrm{even}:=\Big\{u\in C^2_{\mathrm{per}}\big([0,T],\mathds R\big):u^-=u\Big\}.
\end{equation}
Note that the constant solution $u_0\equiv 1$ to \eqref{eq:eqOlambdaT} belongs to $C^2_{\mathrm{per}}\big([0,T],\mathds R\big)^\mathrm{even}$.

Recalling that the set of (positive) solutions to \eqref{eq:bifODE} is invariant by translations, we define two solutions to \eqref{eq:bifODE} to be \emph{equivalent} if they are obtained one from another by a translation.
An easy consequence of periodicity is the following:

\begin{lemma}\label{thm:equivtoeven}
Any (positive) solution to \eqref{eq:bifODE} in $C^2_{\mathrm{per}}\big([0,T],\mathds R\big)$ is equivalent to a (positive) solution to \eqref{eq:bifODE} in $C^2_{\mathrm{per}}\big([0,T],\mathds R\big)^\mathrm{even}$.
\end{lemma}

\begin{proof}
Given any (positive) solution $u$ to \eqref{eq:bifODE} in $C^2_{\mathrm{per}}\big([0,T],\mathds R\big)$, Rolle's Theorem yields a $t_0\in[0,T]$ such that $u'(t_0)=0$. Set $v(t)=u(t-T/2+t_0)$, $t\in\mathds R$. Then, $v$ is a solution to \eqref{eq:bifODE} in  $C^2_{\mathrm{per}}\big([0,T],\mathds R\big)$ equivalent to $u$, which satisfies $v'(T/2)=0$. Since $v^-$ is a solution to \eqref{eq:bifODE} that satisfies $v^-(T/2)=v(T/2)$ and ${(v^-)}'(T/2)=0=v'(T/2)$, we have that $v^-=v$, so $v\in C^2_{\mathrm{per}}\big([0,T],\mathds R\big)^\mathrm{even}$.
\end{proof}

By Lemma~\ref{thm:equivtoeven}, it suffices to count pairwise nonequivalent positive solutions to \eqref{eq:bifODE} in $C^2_{\mathrm{per}}\big([0,T],\mathds R\big)^\mathrm{even}$ in order to establish the desired multiplicity result.
This reduction is key to use bifurcation from \emph{simple} eigenvalues (\Cref{thm:CraRab}).

\subsection{Linearization}\label{sub:linearization}
The linearization around the solution $u_0\equiv1$ of the problem $\mathbf P_{\mu}$, see \eqref{eq:problemOlambdaT}, gives the following elementary linear boundary value problem:
\begin{equation*}
\left\{\begin{aligned}
&v''+(q-1)\mu\, v=0,\\
&v(0)=v(T),\quad v'(0)=v'(T),
\end{aligned}\right.
\end{equation*}
which admits nontrivial solutions if and only if $\mu$ assumes one of the values
\begin{equation}\label{eq:deginstants}
\phantom{,\quad k\in\mathds N.}\mu_k=\dfrac{4\pi^2k^2}{T^2(q-1)},\quad k\in\mathds N.
\end{equation}
These are hence the \emph{degeneracy instants} of $\mathbf P_{\mu}$, away from which there is rigidity:

\begin{proposition}\label{thm:locrigidity}
For all $\mu_*\not\in\{\mu_k:k\in\mathds N\}$, the constant function $u_0\equiv 1$ is a \emph{locally rigid} solution to $\mathbf P_{\mu}$ near $\mu_*$, i.e., there exists a neighborhood $U$ of $(\mu_*,u_0)$ in $\left(0,+\infty\right)\times C^2_\mathrm{per}\big([0,T],\mathds R\big)$ such that if $(\mu,u)\in  U$ and $u$ is a solution to $\mathbf P_{\mu}$, then $u=u_0$.
\end{proposition}

\begin{proof}
This is obtained as an easy application of the Inverse Function Theorem. More precisely, consider the smooth map 
\begin{equation}\label{eq:Cper+20}
C^2_{\mathrm{per}}\big([0,T],\mathds R\big)_+\ni u\longmapsto u''+\mu(u^q-u)\in C^0\big([0,T],\mathds R\big),
\end{equation} 
where $C^2_{\mathrm{per}}\big([0,T],\mathds R\big)_+$ is the open set of positive functions in $C^2_{\mathrm{per}}\big([0,T],\mathds R\big)$. The Frechet derivative of \eqref{eq:Cper+20} at the point $u_0\equiv1$ is the bounded linear map $C^2_{\mathrm{per}}\big([0,T],\mathds R\big)\ni v\mapsto v''+(q-1)\mu\, v\in C^0\big([0,T],\mathds R\big)$, which is a Fredholm map of index $0$. Indeed, it is a compact perturbation of the bounded linear map $C^2_{\mathrm{per}}\big([0,T],\mathds R\big)\ni v\mapsto v''\in C^0\big([0,T],\mathds R\big)$, whose kernel has dimension $1$ and image has codimension $1$.
The condition that $\mu_*\not\in\{\mu_k:k\in\mathds N\}$ gives precisely that, for $\mu=\mu_*$, such linear map is injective, hence an isomorphism. 
\end{proof}

The \emph{eigenvalues} of the linearized problem are defined as those $\lambda\in\mathds R$ for which there exists a nontrivial solution $v$ to the linear boundary value problem
\begin{equation}\label{eq:eigenvlinearization}
\left\{\begin{aligned}
&v''+(q-1)\mu \, v=-\lambda\,v,\\
&v(0)=v(T),\quad v'(0)=v'(T).
\end{aligned}\right.
\end{equation}
Again, an elementary computation shows that, for all $\mu>0$, these eigenvalues form a strictly increasing unbounded sequence $\{\lambda_k(\mu)\}_{k\ge0}$, given by:
\begin{equation}\label{eq:lambdakmu}
\lambda_k(\mu)=\frac{4\pi^2}{T^2}k^2-(q-1)\mu,
\end{equation}
whose corresponding eigenspace is spanned by the eigenfunctions
\begin{equation}\label{eq:correigenfunctions}
\phantom{\omega_k=\frac{2\pi k}T.}v^\mathrm{even}_k(t)=\cos(\omega_k\,t),\quad v^\mathrm{odd}_k(t)=\sin(\omega_k\,t),\qquad \omega_k=\frac{2\pi k}T.
\end{equation}
Clearly, $v^\mathrm{even}_k\in C^2_{\mathrm{per}}\big([0,T],\mathds R\big)^\mathrm{even}$, but $v^\mathrm{odd}_k\notin C^2_{\mathrm{per}}\big([0,T],\mathds R\big)^\mathrm{even}$ for all $k\in\mathds N$, thus all eigenvalues of the restriction of \eqref{eq:eigenvlinearization} to $C^2_{\mathrm{per}}\big([0,T],\mathds R\big)^\mathrm{even}$ are \emph{simple}.

Each $\lambda_k$ is a strictly decreasing function of $\mu$, since
\begin{equation}\label{eq:dereigenvalue}
\lambda_k'(\mu)=-(q-1)<0,
\end{equation}
and, evidently, $\lambda_k(\mu_k)=0$ for all $k\ge0$.

\subsection{Local bifurcation}\label{sub:localbif}
If $\mu\le\frac{4\pi^2}{T^2(q-1)}$, then problem $\mathbf P_{\mu}$, given in \eqref{eq:problemOlambdaT}, has a unique positive solution, given by the constant function $u_0\equiv1$, as a consequence of Theorem~\ref{thm:uniqepsismall}. We are now ready to prove that, if $\mu>\frac{4\pi^2}{T^2(q-1)}$, then problem $\mathbf P_{\mu}$ admits multiple positive solutions in $C^2_{\mathrm{per}}\big([0,T],\mathds R\big)^\mathrm{even}$, which bifurcate from $u_0$.

Following the notation of \Cref{sec:crash}, let $I=(0,+\infty)$, $X=C^2_{\mathrm{per}}\big([0,T],\mathds R\big)^\mathrm{even}$, $Y=C^0\big([0,T],\mathds R\big)$, and set $\mathcal P\colon I\times X\to Y$ to be $\mathcal P(\mu,x)=x''-\mu(x-x\vert x\vert^{q-1})$. Then, the set of solutions $\mathcal S=\mathcal P^{-1}(0)$, as in \eqref{eq:defG}, is given by
\[
\mathcal S =\big\{(\mu,u)\in I\times X :u\ \text{is a solution to}\ \mathbf P_{\mu}\big\},
\]
and it contains the \emph{trivial branch} determined by the constant solution $x_0=u_0$, i.e.,
\[\mathcal B_{\mathrm{triv}}=\big\{(\mu,u_0):\mu>0\big\}\subset\mathcal S.\]

The next result describes the geometry of $\mathcal S$ near the points $(\mu_k,u_0)\in\mathcal B_{\mathrm{triv}}$.

\begin{theorem}\label{thm:localbif}
For all $k\in\mathds N$, the point $(\mu_k,u_0)$ is a bifurcation point for the equation $\mathcal P(\mu,x)=0$, where $\mu_k$ is given by \eqref{eq:deginstants}.
More precisely, for all $k\in\mathds N$, there exists a neighborhood $U_k$ of $(\mu_k,u_0)$ in $I\times X$ such that  $\mathcal S \cap U_k$ consists of the union of $U_k\cap\mathcal B_{\mathrm{triv}}$ and the image of a real analytic path $(-\varepsilon,\varepsilon)\ni s\mapsto(\mu_s,u_s)$ that crosses $\mathcal B_{\mathrm{triv}}$ transversely at $(\mu_k,u_0)$ when $s=0$.
\end{theorem}

\begin{proof}
This is an application of the Crandall--Rabinowitz Theorem~\ref{thm:CraRab}, with the setup described above.
The Fredholmness assumption (a) of Theorem~\ref{thm:CraRab} is easily verified as in the proof of Proposition~\ref{thm:locrigidity}.
Assumption (b) of Theorem~\ref{thm:CraRab} is verified using the fact that the eigenvalues of problem \eqref{eq:eigenvlinearization} in $C^2_{\mathrm{per}}\big([0,T],\mathds R\big)^\mathrm{even}$ are simple, as we observed in Subsection~\ref{sub:linearization} above.
Finally, assumption (c) of Theorem~\ref{thm:CraRab} is easily verified using \eqref{eq:dereigenvalue}. Namely, set $\xi_{k}=v_k^\mathrm{even}$, as in \eqref{eq:correigenfunctions}, and note that $\frac{\partial\mathcal P}{\partial x}(\mu,x_0)\xi_{k}=\lambda_k(\mu)\,\xi_k$ for all $\mu$ near $\mu_k$.
Note also that $\xi_k$ spans the one-dimensional kernel of $\frac{\partial\mathcal P}{\partial x}(\mu_k,x_0)$, and that the image of this map does not contain any nonzero multiple of $\xi_{k}$, since it is $L^2$-orthogonal to $\xi_{k}$. Then:
\[ \frac{\mathrm d}{\mathrm d\mu}\Big\vert_{\mu=\mu_k}\frac{\partial\mathcal P}{\partial x}(\mu,x_0)\xi_{k}=\lambda'_k(\mu_k)\,\xi_k
\in\mathds R\cdot\xi_k\setminus\{0\},\]
so assumption (c) in Theorem~\ref{thm:CraRab} holds at each $\mu_k$, concluding the proof.
\end{proof}

\subsection{Global bifurcation results}\label{sub:globalbif}
By \Cref{thm:locrigidity} and Theorem~\ref{thm:localbif}, the closure $\overline{\mathcal S\setminus\mathcal B_{\mathrm{triv}}}$ of the set $\mathcal S\setminus\mathcal B_{\mathrm{triv}}$ is given by the union of
$\mathcal S \setminus\mathcal B_{\mathrm{triv}}$ and the sequence $(\mu_k,u_0)$, $k\ge1$, of bifurcation points.

\begin{definition}\label{thm:defbifbranch}
For all $k\in\mathds N$, let $\mathcal B_k$ be the connected component of $\overline{\mathcal S \setminus\mathcal B_{\mathrm{triv}}}$ that contains the point $(\mu_k,u_0)$, which we call the $k$th \emph{bifurcation branch} of solutions.
\end{definition}

Applying the Rabinowitz Theorem~\ref{thm:rabinowitz}, we obtain the following:

\begin{proposition}\label{thm:globbif}
For all $k\in\mathds N$, the $k$th bifurcation branch $\mathcal B_k$ either contains a point $(\mu_{k'},u_0)\in \mathcal B_{\mathrm{triv}}$, with $k'\neq k$, or else is noncompact. 
\end{proposition}

\begin{proof}
This is easily obtained from Theorem~\ref{thm:rabinowitz} and Corollary~\ref{thm:properprojection}. 
The only detail that requires some justification is the positivity condition $u>0$ in $\mathbf P_{\mu}$, since, in principle, the bifurcation branches for the boundary value problem 
\begin{equation}\label{eq:withoutpositivity}
u''-\mu(u-\vert u\vert^{q-1}u)=0,\quad u(0)=u(T),\quad u'(0)=u'(T),
\end{equation}
to which Theorem~\ref{thm:rabinowitz} applies, may leave the open set defined by $u>0$. However, Proposition~\ref{thm:positivityinabranch} tells us that this does not occur: all solutions to \eqref{eq:withoutpositivity} in bifurcation branches issuing from $\mathcal B_{\mathrm{triv}}$ remain positive on $[0,T]$.
\end{proof}

Let us show that the first alternative in the statement of Proposition~\ref{thm:globbif} does not occur. This is done by showing that bifurcation branches do not intersect, which implies that each bifurcation branch contains exactly one point in the trivial branch.

\begin{proposition}\label{thm:numberofzerosbranch}
If $(\mu,u)\in\mathcal B_k \setminus\mathcal B_{\mathrm{triv}}$, then $u-u_0$ has exactly $2k$ zeros in $\left[0,T\right)$.
\end{proposition}

\begin{proof}
For all $(\mu,u)\in\mathcal S \setminus\mathcal B_{\mathrm{triv}}$, the function $u-u_0$ has only simple zeros by Corollary~\ref{thm:qualitative}. The $T$-periodic function $(u-u_0)\colon\mathds R\to\mathds R$ has the same number of zeros in $\left[0,T\right)$ as the function $(u-u_0)\colon \mathds S^1\to \mathds R$ it induces on the quotient $\mathds S^1=\mathds R/T\mathds Z$. Clearly, a continuous family of real-valued functions on $\mathds S^1$ whose zeros are simple must have a constant number of zeros. This applies to each one of the (at most two) connected components of
$\mathcal B_k \setminus\{(\mu_k,u_0)\}$.
So, in order to conclude the proof, it suffices to show that, given any $(\mu,u)\in\mathcal B_k \setminus\{(\mu_k,u_0)\}$ sufficiently close to $(\mu_k,u_0)$, then $u-u_0$ has exactly $2k$ zeros in $\left[0,T\right)$. This follows from Theorem~\ref{thm:CraRab}, since it implies that, near each bifurcation point $(\mu_k,u_0)$, the bifurcation branch is the image of a real analytic path 
\[\left(-\varepsilon,\varepsilon\right)\ni s\longmapsto\big(\mu_k(s),1+s\,J_k+s\,\psi_k(s)\big),\] 
where $\mu_k(0)=\mu_k$, $J_k(t)=\cos\!\big(\frac{2\pi k}{T}t\big)$, and $\psi_k(s)\in C^2_{\mathrm{per}}\big([0,T],\mathds R\big)^\mathrm{even}$ for all $s$, with $\psi_k(0)=0$. 
Thus, if $(\mu,u)\in\mathcal B_k \setminus\{(\mu_k,u_0)\}$ is sufficiently close to $(\mu_k,u_0)$, then zeros of $u-u_0$ correspond to zeros of $f^s(t)=\cos\!\big(\frac{2\pi k}{T}t\big)+s \psi_k(s)(t)$. Finally, if $\vert s\vert$ is small enough, then the  functions $f^s(t)$ and $\cos\!\big(\frac{2\pi k}{T}t\big)$ have the same number of zeros in $[0,T)$, which is equal to $2k$.
\end{proof}

\begin{corollary}\label{thm:noncompactbranches}
The bifurcation branches $\mathcal B_k$  are pairwise disjoint and noncompact.
\end{corollary}

\begin{proof}
Distinct bifurcation branches do not intersect at any point outside $\mathcal B_{\mathrm{triv}}$ by Proposition~\ref{thm:numberofzerosbranch}, since they correspond to functions with different numbers of zeros in $[0,T)$. Furthermore, distinct bifurcation branches cannot intersect at a point on $\mathcal B_{\mathrm{triv}}$, because, by the local form of the bifurcation branches (see Theorem~\ref{thm:CraRab}), there is only one bifurcation branch near each bifurcation point along $\mathcal B_{\mathrm{triv}}$. Thus, each bifurcation branch contains exactly one point of the trivial branch $\mathcal B_{\mathrm{triv}}$, so it is noncompact by Proposition~\ref{thm:globbif}.
\end{proof}

Let us now show that each bifurcation branch $\mathcal B_k$ contains points $(\mu,u)$ with arbitrary $\mu\ge\mu_k$. 
Recall $\Pi\colon I\times X\to I$ is the projection onto the first factor. 

\begin{proposition}\label{thm:arbitrarymu}
For all $k\in\mathds N$, the set $\Pi\big(\mathcal B_k \setminus\{(\mu_k,u_0)\}\big)$ contains the open half-line $\left(\mu_k,+\infty\right)$.
\end{proposition}
\begin{proof}
Let us show that $\Pi\big(\mathcal B_k \big)\supset\left[\mu_k,+\infty\right)$. Clearly, $\mu_k\in \Pi\big(\mathcal B_k \big)$. Since $\Pi\big(\mathcal B_k \big)$ is a connected subset of $\mathds R$, it suffices to show that $\Pi\big(\mathcal B_k \big)$ is unbounded. If this were not the case, there would exist $\mu_*>0$ such that $\mathcal B_k \subset{\Pi}^{-1}\big([\mu_1,\mu_*]\big)$. However, by Corollary~\ref{thm:properprojection}, the restriction of $\Pi$ to $\mathcal S$ is proper, and therefore ${\Pi}^{-1}\big([\mu_1,\mu_*]\big)\cap\mathcal S$ is compact. But $\mathcal B_k$ is closed in $\mathcal S$, and it is noncompact by Corollary~\ref{thm:noncompactbranches}. This gives a contradiction originating from the assumption that $\Pi\big(\mathcal B_k\big)$ is bounded.
\end{proof}

Together, Theorem~\ref{thm:uniqepsismall}, \Cref{thm:numberofzerosbranch,thm:arbitrarymu}, and \Cref{thm:noncompactbranches} lead to the following conclusion, proving Theorem~\ref{mainthm:A} in the Introduction, see also \Cref{fig:bif-branches}.

\begin{theorem}\label{thm:finalresult}
The following hold, where $\mu_k=\frac{4\pi^2k^2}{T^{2}(q-1)}$, $k\in\mathds N$,
as in \eqref{eq:deginstants}.
\begin{enumerate}[\rm (a)]
\item
For all $\mu\in\left(0,\mu_1\right]$, problem $\mathbf P_{\mu}$ has a unique solution, given by $u\equiv1$;
\item
For all $k\in\mathds N$ and $\ell\in\{1,\ldots,k\}$,\ if $\mu\in\left(\mu_k,\mu_{k+1}\right]$, then there exists a solution $u_\ell\in X$ of problem $\mathbf P_{\mu}$, such that $u_{\ell}-1$ has exactly $2\ell$ zeros in $\left[0,T\right)$.
\end{enumerate}
\end{theorem}

\begin{figure}[htbp]
\vspace{-.25cm}
\centering
\begin{tikzpicture}[scale=1.07]
\fill[black!50,opacity=0.25] (-1.2,3.5) -- (1.2,2) -- (1.2,-3.5) -- (-1.2,-2) -- cycle;
\draw (.5,-2.4) node {$\{\mu\}\times X$};
\draw (-5,-2.6) node[left] {$\phantom{C^2_{\mathrm{per}}\big([0,T],\mathds R\big)^\mathrm{even}}$};

  \fill[top color=gray!15,bottom color=gray!10,middle color=gray!5,shading=axis,opacity=0.15] (0,0) 
  circle[x radius=2,y radius=0.5,rotate=90];
  \pgfmathsetmacro{\alphacrit}{acos(1/12)}  
  \draw[left color=gray!30!black,right color=gray!10!black,middle color=gray!20,shading=axis,
    fill opacity=0.25, thick,rotate=90] 
    (90-\alphacrit:2 and 0.5) -- (0,6) -- (90+\alphacrit:2 and 0.5) 
  arc[start angle=90+\alphacrit,end angle=450-\alphacrit,x radius=2,y radius=0.5,rotate=90];
  \draw[densely dashed, thick,rotate=90] 
    (90-\alphacrit:2 and 0.5) 
  arc[start angle=90-\alphacrit,end angle=90+\alphacrit,x radius=2,y radius=0.5,rotate=90];
  
  \draw[thick, -] (-6,-1) -- (-3.01,-1);
  \draw[thick, dashed, -] (-2.91,-1) -- (0.4,-1);
  \draw[thick, ->] (0.425,-1) -- (2,-1) node[right] {${\mu}$};
  
  \draw[-] (-6,-0.9) -- (-6,-1.1);
  \draw (-6,-1.155) node[below] {$_{0}$};
  \draw[-] (-4.45,-0.9) -- (-4.45,-1.1);
  \draw (-4.45,-1.155) node[below] {$_{\mu_1}$};
  \draw[-] (-2.86,-0.9) -- (-2.86,-1.1);
  \draw (-2.86,-1.155) node[below] {$_{\mu_2}$};
  \draw[-] (-1.38,-0.9) -- (-1.38,-1.1);
  \draw (-1.38,-1.155) node[below] {$_{\mu_3}$};
  % \draw[-] (-0.12,-0.9) -- (-0.12,-1.1);
  % \draw (-0.12,-1.155) node[below] {$_{\mu_4}$};

\draw[myred,very thick,variable=\t,domain=0:1.21,samples=400,rotate=90]
      plot ({sinh(\t)},{10-5.5*cosh(\t)});
  \draw[myred,very thick,variable=\t,domain=0:1.21,samples=400,rotate=90]
      plot ({-sinh(\t)},{10-5.5*cosh(\t)});

  \draw[myred,very thick,variable=\t,domain=0:1.045,samples=400,rotate=90]
      plot ({sinh(\t)},{8-5*cosh(\t)});
  \draw[myred,very thick,variable=\t,domain=0:1.045,samples=400,rotate=90]
      plot ({-sinh(\t)},{8-5*cosh(\t)});

  \draw[myred,very thick,variable=\t,domain=0:.79,samples=400,rotate=90]
      plot ({sinh(\t)},{6-4.5*cosh(\t)});
  \draw[myred,very thick,variable=\t,domain=0:.79,samples=400,rotate=90]
      plot ({-sinh(\t)},{6-4.5*cosh(\t)});

  % \draw[myred,very thick,variable=\t,domain=0:.37,samples=400,rotate=90]
  %     plot ({sinh(\t)},{4.25-4*cosh(\t)});
  % \draw[myred,very thick,variable=\t,domain=0:.37,samples=400,rotate=90]
  %     plot ({-sinh(\t)},{4.25-4*cosh(\t)});    

\draw (-6,0) node[left] {${\color{myblue} \mathcal B_{\mathrm{triv}}}$};
\draw[myblue,very thick,variable=\t,domain=0:5.98,samples=200,rotate=90]
      plot (0,\t);

 \draw[thick,rotate=90] 
    (90-\alphacrit:2 and 0.5) -- (0,6) -- (90+\alphacrit:2 and 0.5) 
  arc[start angle=90+\alphacrit,end angle=450-\alphacrit,x radius=2,y radius=0.5,rotate=90];

  \draw[fill=white] (-4.5,0) circle[radius=1pt];
  \draw[fill=white] (-3,0) circle[radius=1pt];
  \draw[fill=white] (-1.5,0) circle[radius=1pt];
  % \draw[fill=white] (-0.25,0) circle[radius=1pt];

  \draw (0.15,1.65) node {$_{u_1}$};
  \draw[fill=myred] (0,1.52) circle[radius=1pt];
  \draw[fill=myred] (0,-1.52) circle[radius=1pt];

  \draw (0.15,1.34) node {$_{u_2}$};
  \draw[fill=myred] (0,1.25) circle[radius=1pt];
  \draw[fill=myred] (0,-1.25) circle[radius=1pt];

  \draw (0.15,.98) node {$_{u_3}$};
  \draw[fill=myred] (0,0.88) circle[radius=1pt];
  \draw[fill=myred] (0,-0.88) circle[radius=1pt];

  % \draw (0.15,0.48) node {$_{u_4}$};
  % \draw[fill=myred] (0,0.36) circle[radius=1pt];
  % \draw[fill=myred] (0,-0.36) circle[radius=1pt];

  \draw (0.15,0.2) node {$_{u_0}$};
  \draw[fill=myblue] (0,0) circle[radius=1pt];

\draw (-3.8,0.25) node {${\color{myred} \mathcal B_1}$};
\draw (-2.3,0.25) node {${\color{myred} \mathcal B_2}$};
\draw (-0.8,0.25) node {${\color{myred} \mathcal B_3}$};
\end{tikzpicture}
\vspace{-.25cm}
\caption{A portion of the set $\mathcal S$ of solutions to $\mathbf P_{\mu}$, with three bifurcation branches (red) issuing from the trivial branch (blue). By \Cref{thm:positivityinabranch}, solutions in bifurcation branches are positive; by \Cref{thm:uniformbounds}, their $C^2$-norm grows at most linearly with~$\mu$.}
\label{fig:bif-branches}
\end{figure}
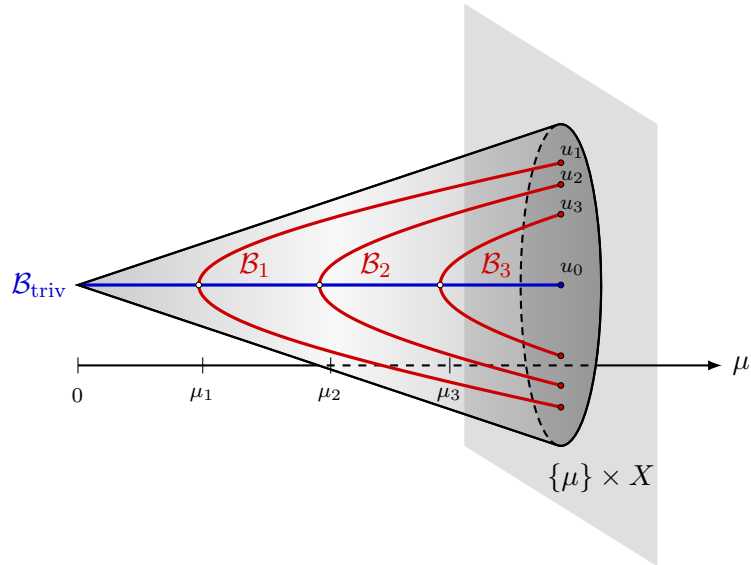 

\begin{remark}
Each bifurcation branch $\mathcal B_k$ contains the image of a real analytic path $(-\infty,+\infty) \ni s\mapsto \big(\mu_k(s),u_k(s)\big)\in I\times X$, by \Cref{rem:other-side}. Thus, in light of \Cref{thm:uniqepsismall}, \Cref{thm:properprojection}, and \Cref{thm:arbitrarymu}, it intersects each slice $\{\mu\}\times X$ with $\mu>\mu_k$ at least \emph{twice}, see \Cref{fig:bif-branches}. However, it is unclear to us whether such solutions must be equivalent (i.e., obtained from one another by a translation).
\end{remark}

\end{section}

\begin{section}{The Yamabe equation}
\subsection{Conformal metrics and scalar curvature}
Background material regarding this subsection can be found in the standard texts \cite{LeeParker87} and \cite{Schoen89}.
\smallskip

Let $(M^{n+1},g_M)$, $n\ge2$, be a compact Riemannian manifold with constant scalar curvature $R_M>0$.
Given a smooth positive function $u\colon M\to\mathds R$, the conformal metric $g=u^\frac4{n-1}\cdot g_M$ has scalar curvature $R_g$ given by
\[R_g=\frac{4n}{n-1}u^{-\frac{n+3}{n-1}}\,L_M(u),\]
where $L_M$ is the \emph{conformal Laplacian} of $(M,g_M)$, defined as
\[L_M(u)=-\Delta_M(u)+\frac{n-1}{4n}R_M\,u,\]
where $\Delta_M$ is the (nonpositive) Laplacian of $(M,g_M)$. Recall that the volume of $(M,g)$ is $\vol(M,g)=\int_Mu^{\frac{2n+2}{n-1}}\,\dd\nu_M$,
where $\dd\nu_M$ is the volume element of $g_M$.

The \emph{Yamabe problem} consists of finding constant scalar curvature metrics in the conformal class of $g_M$. Using the above setup, this is equivalent to determining solutions to the following nonlinear elliptic PDE on $M$:
\begin{equation}\label{eq:Yamabe1}
-\Delta_M(u)+\frac{n-1}{4n}\,R_M\, u=C\,\frac{n-1}{4n}\,u^{\frac{n+3}{n-1}},
\end{equation}
for some (necessarily positive) constant $C$. Homothetic solutions to the problem are considered equivalent. So, it is customary to either impose a volume constraint for solutions, or, as we will do here, to normalize the value $R_M$ of the scalar curvature of solutions. This means that we set $C=R_M$ in \eqref{eq:Yamabe1}, and consider all the solutions to
\begin{equation}\label{eq:Yamabe2}
\Delta_M(u)-\frac{n-1}{4n}\,R_M \, \left[u-u^{\frac{n+3}{n-1}}\right]=0.
\end{equation}  
Clearly, the constant function $u\equiv1$ solves \eqref{eq:Yamabe2}.

\subsection[The Yamabe equation]{\texorpdfstring{The Yamabe equation in $N\times\mathds S^1$}{The Yamabe equation}}
Let us now consider a closed Riemannian manifold $(N^n,g_N)$, $n\ge2$, with constant scalar curvature $R_N>0$, a positive number $r$, and the circle $\mathds S^1=\mathds R/2\pi\mathds Z$ endowed with the metric $r^2\dd t^2$ of length $2\pi r$. Then, we set $M=N\times\mathds S^1$, and consider the product metric $g_M=g_N\oplus r^2\dd t^2$ on $M$, whose scalar curvature is equal to $R_N$. Let us restrict ourselves to conformal factors $u\colon N\times\mathds S^1\to\mathds R^+$ that depend only on the variable $t\in\mathds S^1$. In this situation, the Laplacian $\Delta_M(u)$ coincides with the Laplacian of the circle $\mathds S^1$, given by
\[\Delta_M(u)=\frac1{r^2}u'',\]
and the Yamabe equation \eqref{eq:Yamabe2} becomes:
\begin{equation}\label{eq:Yamabe3}
u''-\frac{n-1}{4n}R_N\,r^2\left(u-u^{\frac{n+3}{n-1}}\right)=0.
\end{equation}
Thus, these so-called \emph{basic} solutions to the Yamabe problem correspond to the functions $u\colon\mathds R\to\mathds R^+$ that solve \eqref{eq:Yamabe3} and satisfy the periodicity conditions
\begin{equation}\label{eq:BCYamabe}
u(0)=u(2\pi),\quad\text{and}\quad u'(0)=u'(2\pi).
\end{equation}
In other words, the \emph{basic Yamabe problem} on $N\times \mathds S^1$, given by \eqref{eq:Yamabe3} and \eqref{eq:BCYamabe}, reduces precisely to problem $\mathbf P_{\mu}$ given in \eqref{eq:problemOlambdaT}, with $q=\frac{n+3}{n-1}>1$, $T=2\pi$, and $\mu=\frac{n-1}{4n}R_N\,r^2>0$. Replacing the parameter $\mu>0$ with the parameter $r>0$, the degeneracy instants $\mu_k=\frac{4\pi^2k^2}{T^2(q-1)}$  correspond to $r_k=k\sqrt{n/R_N}$, see \eqref{eq:deginstants}.

More generally, given an isometry $\phi\in\mathrm{Iso}(N,g_N)$, consider the \emph{mapping torus} $M_\phi:=N\times [0,2\pi]/\sim$, where $(p,0)\sim (\phi(p),2\pi)$ for all $p\in N$. 
Since $\phi\colon N\to N$ is an isometry, the product metric $g_N\oplus r^2\dd t^2$ on $N\times [0,2\pi]$ descends to a (locally isometric) Riemannian metric on $M_\phi$, such that the projection $M_\phi\to \mathds S^1$ is a Riemannian submersion with totally geodesic fibers isometric to $(N,g_N)$. In particular, basic solutions to the Yamabe problem on $M_\phi$ also correspond to $u\colon\mathds R\to\mathds R^+$ satisfying $\mathbf P_{\mu}$, as above.
While the mapping torus $M_{\phi}$ is diffeomorphic to $N\times \mathds S^1$ if $\phi\colon N\to N$ is smoothly isotopic to the identity, this construction produces different manifolds if 
 $\phi\in\mathrm{Iso}(N,g_N)$ is in different (nontrivial) smooth isotopy classes.

\subsection{Multiplicity of solutions to the Yamabe problem}
In the context of the basic Yamabe problem on $N\times\mathds S^1$, or, more generally, $M_\phi$, Theorem~\ref{thm:finalresult} yields the following result, which implies Theorem~\ref{mainthmB} in the Introduction:

\begin{theorem}\label{thm:finalresult_yamabe}
Let $(N^n,g_N)$ be a compact Riemannian manifold with constant scalar curvature $R_N>0$ and $n\ge2$.
Consider the Yamabe problem on the product manifold $M=N\times\mathds S^1$; or, more generally, on the mapping torus $M=M_\phi$, where $\phi\in\mathrm{Iso}(N,g_N)$.
The number of basic solutions in the conformal class of $g_N\oplus r^2\dd t^2$ 
is equal to $1$ if $0<r<\sqrt{n/R_N}$, and becomes arbitrarily large as~$r\nearrow+\infty$.
More precisely, given $k\in\mathds N$ and $r>k\,\sqrt{n/R_N}$, there exist smooth positive nonconstant functions $u_{\ell} \colon M\to\mathds R^+$ for each $\ell\in\{1,\ldots,k\}$,
which factor through the projection $M\to \mathds S^1$, in such way that $(u_{\ell}-1)\colon \mathds S^1\to \mathds R$ has exactly $2\ell$ zeros on $\mathds S^1$, and the conformal metric $u_{\ell}^\frac4{n-1}(g_N\oplus r^2\dd t^2)$ has constant scalar curvature equal to $R_N$.
\end{theorem}

\begin{proof}
All claims follow from Theorem~\ref{thm:finalresult}, applied to the Yamabe equation \eqref{eq:Yamabe3} with periodic boundary condition~\eqref{eq:BCYamabe}. 
Namely, $u_\ell$ can be taken so that $(\mu, u_\ell)$ is any point in $\big(\mathcal B_\ell \setminus\mathcal  B_{\mathrm{triv}}\big)\cap \big(\{\mu\}\times C^2_{\mathrm{per}}\big([0,2\pi],\mathds R\big)^\mathrm{even}\big)$. By Proposition~\ref{thm:arbitrarymu}, the intersections above are nonempty whenever $\mu=\frac{n-1}{4n}\,R_N\, r^2>\frac{\ell^2}{q-1}=\frac{n-1}4\ell^2,$ i.e., 
$r>\ell\sqrt{ n/R_N}$. For any $k\in\mathds N$, if $r>k\sqrt{ n /R_N}$, then this is evidently the case for all $\ell\in\{1,\ldots,k\}$.
\end{proof}

\begin{remark}
A weaker (but easier to prove) nonuniqueness result for solutions to the Yamabe problem on manifolds as in Theorem~\ref{thm:finalresult_yamabe} follows from combining 
the existence of \emph{Yamabe metrics} (i.e., solutions to \eqref{eq:Yamabe1} that minimize the corresponding energy functional in its conformal class) with Aubin's inequality for the Yamabe invariant. More precisely, since there exist degree $k$ coverings $\mathds S^1\to\mathds S^1$ for all $k\in\mathds N$, the pullback of a Yamabe metric on $M=N\times \mathds S^1$ via the corresponding coverings of $M$ cannot remain Yamabe for $k$ sufficiently large, since its large volume would contradict Aubin's inequality, see~\cite{LeeParker87}. Iterating this procedure gives arbitrarily many solutions conformal to $g_N\oplus r^2\dd t^2$ provided $r$ is sufficiently large.
This is also the key idea behind more general multiplicity results in \cite{aif}.
\end{remark}

\begin{remark}\label{rem:GidasNiNirenberg}
Using a celebrated result of Gidas, Ni, and Nirenberg~\cite{GidNiNir79}, it is well-known that, if $(N,g)$ is the round sphere $(\mathds S^n,g_\text{round})$, then all the constant scalar curvature metrics on $\mathds S^n\times\mathds S^1$ that belong to the conformal class of the product metric $g_\text{round}\oplus r^2\dd t^2$ are obtained from a \emph{basic} conformal factor. Thus, in this case, Theorem~\ref{thm:finalresult_yamabe} classifies \emph{all} solutions to the Yamabe problem in the conformal class of $g_\text{round}\oplus r^2\dd t^2$. In particular, by Theorem~\ref{thm:uniqepsismall}, if $r$ is sufficiently small, then all constant scalar curvature metrics in the conformal class of $g_\text{round}\oplus r^2\dd t^2$ are homothetic to $g_\text{round}\oplus r^2\dd t^2$.
\end{remark}

\begin{remark}\label{thm:remgeomdistinct}
A different (and interesting) question is to determine whether two distinct solutions to the Yamabe equation produce \emph{nonisometric} constant scalar curvature metrics, see e.g.~\cite[Rem.~4.1]{BetPicSan16} or \cite[Rem.~2.2]{aif}.
For instance, even  though the Yamabe equation on the round sphere $\mathds S^n$ admits a noncompact set of solutions, all such constant scalar curvature metrics are pairwise isometric. 
In particular, it seems reasonable to expect that the metrics $u_{\ell}^\frac4{n-1}(g_N\oplus r^2\dd t^2)$ in \Cref{thm:finalresult_yamabe} are \emph{nonisometric} for \emph{different} values of $\ell$.

This is the case, for instance, if $(N^n,g_N)$ is the round sphere $(\mathds S^n,g_\text{round})$, with $n\geq 2$. In fact, metrics conformal to $g_\text{round}\oplus r^2\dd t^2$ on $\mathds S^n\times\mathds S^1$ lift to metrics conformal to $g_\text{round}\oplus \dd t^2$ on $\mathds S^n\times\mathds R$, and hence conformal to $g_\text{round}$ on the twice punctured sphere $\mathds S^{n+1}\setminus \mathds S^0$. By Liouville's Theorem, all conformal diffeomorphisms of $(\mathds S^{n+1}\setminus \mathds S^0, g_\text{round})$ are M\"obius transformations and hence extend to conformal diffeomorphisms of $(\mathds S^{n+1},g_\text{round})$ that leave $\mathds S^0$ invariant; these correspond exactly to the isometries of $(\mathds S^n\times\mathds R,g_\text{round}\oplus\dd t^2)$, see e.g.~\cite{kulkarni-book}.
Thus, conformal factors $u_{\ell_1}$ and $u_{\ell_1}$ on $\mathds S^n\times\mathds S^1$ that give rise to isometric metrics $g_i=u_{\ell_i}^\frac4{n-1}\big(g_\text{round}\oplus r^2\dd t^2\big)$, $i=1,2$, 
have lifts $\widetilde u_{\ell_i}\colon \mathds S^n\times\mathds R\to\mathds R$ that
satisfy $\widetilde u_{\ell_2}= \widetilde u_{\ell_1}\circ\varphi$, where $\varphi=(\alpha,\beta)\in \mathrm{Iso}(\mathds S^n\times\mathds R,g_\text{round}\oplus\dd t^2)=\mathrm{Iso}(\mathds S^n,g_\text{round})\times \mathrm{Iso}(\mathds R,\dd t^2)$ is the isometry such that $\widetilde g_2=\varphi^*(\widetilde g_1)$. In particular, as $u_{\ell_1}$ and $u_{\ell_2}$ only depend on the variable $t\in\mathds S^1$, they are obtained from one another by composition with an isometry of $(\mathds S^1,r^2\dd t^2)$, so $u_{\ell_1}-1$ and $u_{\ell_2}-1$ must have the same number of zeroes, i.e., $\ell_1=\ell_2$.
\end{remark}
\end{section}

\end{document}